\documentclass[11pt,letterpaper]{amsart}

\usepackage{amsmath}
\usepackage{amsthm}
\usepackage{amssymb}
\usepackage{amsfonts}
\usepackage{amsxtra}
\usepackage{fullpage}
\usepackage{amscd}
\usepackage{color} 
\usepackage{bm} 
\usepackage{mathrsfs}
\usepackage{subfigure}
\usepackage{tikz}
\usepackage{tikz-cd}
\usetikzlibrary{cd}
\usepackage{enumerate}
\usepackage[all]{xy}
\usepackage[mathscr]{eucal}
\usepackage{verbatim}

\let\nc\newcommand
\let\renc\renewcommand

\theoremstyle{plain}
\newtheorem{thm}{Theorem}
\newtheorem{prop}[thm]{Proposition}
\newtheorem{cor}[thm]{Corollary}
\newtheorem{lem}[thm]{Lemma}

\theoremstyle{definition}
\newtheorem{defn}[thm]{Definition}

\numberwithin{thm}{section}

\nc{\bdm}{\begin{displaymath}}
\nc{\edm}{\end{displaymath}}
\nc{\bthm}{\begin{thm}}
\nc{\ethm}{\end{thm}}
\nc{\blem}{\begin{lem}}
\nc{\elem}{\end{lem}}
\nc{\bcor}{\begin{cor}}
\nc{\ecor}{\end{cor}}
\nc{\bprop}{\begin{prop}}
\nc{\eprop}{\end{prop}}
\nc{\bdef}{\begin{defn}}
\nc{\eddef}{\end{defn}}

\makeatletter
\renewcommand{\subsection}{\@startsection{subsection}{2}{0pt}{-3ex
plus -1ex minus -0.2ex}{-2mm plus -0pt minus
-2pt}{\normalfont\bfseries}} \makeatother

\numberwithin{equation}{section}

\DeclareMathOperator{\gr}{\mathrm{gr}}

\newcommand{\beq}{\begin{equation}\label}
\newcommand{\eeq}{\end{equation}}

\nc{\Z}{\mathbb{Z}}

\newcommand{\C}{\mathbb{C}}

\nc{\rank}{\textrm{rank} \,}
\nc{\ds}{\dots}

\let\mc\mathcal
\let\mf\mathfrak

\nc{\mbf}{\mathbf}
\nc{\Res}{\mathsf{Res} \, }
\nc{\Ind}{\mathsf{Ind} \, }
\nc{\cont}{\textrm{cont}}

\nc{\msf}{\mathsf}

\nc{\minusone}{-1}
\nc{\minustwo}{-2}
\nc{\Mod}{\mathrm{Mod} \,}
\nc{\ms}{\mathscr}
\nc{\Frac}{\mathrm{Frac} \,}
\nc{\ra}{\rightarrow}
\nc{\hra}{\hookrightarrow}
\nc{\lab}{\label}
\renc{\O}{\mc{O}}
\nc{\Tan}{\mc{T}}
\nc{\ul}{\underline}
\nc{\s}{\mathfrak{S}}
\nc{\g}{\mf{g}}
\nc{\pa}{\partial}
\nc{\tit}{\textit}
\nc{\Maxspec}{\mathrm{Maxspec} \, }
\nc{\gldim}{\mathrm{gl.dim}}
\nc{\rkm}{\mathrm{rk} \, (\mf{m})}
\nc{\sm}{\mathrm{sm}}
\nc{\PD}{\mathbb{PD}}
\nc{\hilb}{\textrm{Hilb}}
\nc{\<}{\langle}
\renc{\>}{\rangle}
\nc{\T}{\mathbb{T}}
\nc{\X}{\mathbb{X}}
\nc{\F}{\mathbb{F}}
\nc{\id}{\msf{id}}
\nc{\A}{\mathbb{A}}
\nc{\Grat}{\mc{Grat}}
\nc{\Squo}[1]{\A^{(#1)}}
\nc{\twist}{\mathrm{twist}}
\nc{\Cd}{\mc{C}}
\nc{\Span}{\mathrm{Span}}
\nc{\Grass}{\mathrm{Gr}}
\nc{\Supp}{\mathrm{Supp}}
\nc{\Irr}{\mathrm{Irr}}
\renc{\o}{\otimes}
\renc{\gr}{\mathsf{gr}}
\nc{\fin}{\mathrm{fin}}
\nc{\aff}{\mathrm{aff}}
\nc{\algD}{\mf{D}}
\nc{\hr}{\mf{h}_{\textrm{reg}}}
\nc{\D}{\mathscr{D}}
\nc{\PIdeg}{\mathrm{PI-degree}}
\nc{\ch}{\mathrm{ch}}
\nc{\ev}{\mathsf{ev}}
\nc{\Stab}{\mathrm{Stab}}
\nc{\Der}{\mathrm{Der}}
\nc{\rightsim}{\stackrel{\sim}{\longrightarrow}}
\nc{\HZ}{H_{\mbf{h},\Z}(\Z_m)}
\nc{\sing}{\mathrm{sing}}
\nc{\dd}{\mathscr{D}}
\nc{\bc}{\mathbf{c}}
\nc{\vc}{\underline{\mathbf{c}}}
\nc{\ba}{\mathbf{a}}
\nc{\reg}{\mathrm{reg}}
\nc{\Amp}{\mathrm{Amp}}
\nc{\Nef}{\mathrm{Nef}}
\nc{\SL}{\mathrm{SL}}
\nc{\Sp}{\mathrm{Sp}}
\nc{\Sym}{\mathrm{Sym}}
\nc{\Mov}{\mathrm{Mov}}
\nc{\Pic}{\mathrm{Pic}}
\nc{\Cs}{\C^{\times}}
\nc{\Nak}[3]{\mf{M}_{{#1}} ({#2},{#3}) }
\nc{\Naka}[2]{\mf{M}({#1},{#2}) }
\nc{\Mtheta}[1]{\mc{M}_{#1}}

\nc{\bw}{\mathbf{w}}

\nc{\bn}{\mathbf{n}}
\nc{\CB}{\mathrm{CB}}
\nc{\GVect}{\Lambda}
\nc{\pZ}{\overline{Z}}
\nc{\Tang}{\mc{T}}
\nc{\K}{\mathbb{K}}

\newcommand{\mr}{\mathrm}

\newcommand{\git}{\ensuremath{/\!\!/\!}}

\nc{\red}[1]{\textcolor{red}{#1}}

\begin{document}

\title{Coulomb branches have symplectic singularities}

\author[G. Bellamy]{Gwyn Bellamy}
\address{School of Mathematics and Statistics, University of Glasgow, University Place,
Glasgow, G12 8QQ.}
\email{gwyn.bellamy@glasgow.ac.uk}

\begin{abstract}
	We show that Coulomb branches for $3$-dimensional $\mathcal{N}=4$ supersymmetric gauge theories have symplectic singularities. This confirms a conjecture of Braverman-Finkelberg-Nakajima. 
\end{abstract}

\maketitle

\section{Introduction}

Let $G$ be a complex reductive algebraic group and $N$ a finite-dimensional representation of $G$. A mathematical definition of the Coulomb branch (of cotangent type) $\mc{M}_C(G,N)$ of a $3$-dimensional $\mathcal{N}=4$ supersymmetric gauge theory associated to $(G,N)$ was introduced in the seminal papers \cite{BFNI,BFNII}. They showed that Coulomb branches have a number of remarkable properties. Of relevance to us is the fact that they are irreducible normal Poisson varieties, where the Poisson structure is non-degenerate on the smooth locus. Therefore it is natural to conjecture, as they do, that Coulomb branches have symplectic singularities in the sense of Beauville \cite{Beauville}.  

Using partial resolutions of singularities constructed from flavour symmetries, it was shown by Weekes \cite{WeekesSympSing} that most Coulomb branches arising from quiver gauge theories have symplectic singularities. In this note, we extend that result by showing that all Coulomb branches have symplectic singularities.

\begin{thm}\label{thm:rankonemain}
$\mc{M}_C(G,N)$ has symplectic singularities. 
\end{thm}

This confirms the "optimistic conjecture'' of Braverman-Finkelberg-Nakajima \cite[3(iv)]{BFNII}. As an immediate corollary we note that:

\begin{cor}
	$\mc{M}_C(G,N)$ has finitely many symplectic leaves. 
\end{cor}

In the case of quiver gauge theories for finite type quivers, the symplectic leaves of $\mc{M}_C(G,N)$ have been explicitly described in \cite{MuthiahWeeksSlices}. See \cite{WeekesSympSing} for other consequences of the main theorem. 

Our proof relies on an elementary observation about varieties with symplectic singularities. Namely, if there is a birational Poisson morphism $X \to Y$ between normal affine varieties and $Y$ is known to have symplectic singularities then so too does $X$. We apply this observation twice - first in the case where $G$ is a (connected) torus to allow us to reduce to the case where the Coulomb branch can be identify with a toric hyper-K\"ahler manifold and secondly to reduce from the case of a Coulomb branch for a general reductive group to one for a torus. In both cases, the birational Poisson morphism we require was already constructed by Braverman-Finkelberg-Nakajima \cite{BFNII}.



\section{The proof}

\subsection{An elementary observation}

Throughout, variety will mean a integral, separated scheme of finite type over the complex numbers. We recall, following \cite{Beauville}, that a variety $X$ has symplectic singularities if it is a normal variety whose smooth locus admits a symplectic form $\omega$ such that for some (any) resolution of singularities $q \colon Z \to X$, $q^*\omega$ extends to a regular $2$-form on $Z$.

The following elementary lemma is the key to the proof of the main theorem. 

\begin{lem}\label{lem:keydominantsymp}
	Let $X,Y$ be complex normal Poisson varieties. Assume that $Y$ has symplectic singularities and the Poisson structure on the smooth locus of $X$ is non-degenerate. If there exists a generically \'etale Poisson morphism $f \colon X \to Y$ then $X$ has symplectic singularities.   
\end{lem}

\begin{proof}
	The only thing to check is that the pull-back to some resolution of singularities of the symplectic form $\omega$ on the smooth locus of $X$ is regular. Let $\omega_0$ denote the symplectic form on the smooth locus of $Y$. 
	
	We choose a resolution of singularities $p \colon W \to Y$. Let $C$ denote the (unique) irreducible component of $W \times_Y X$ dominating both $Y$ and $X$. By base change, $C \to X$ is a proper generically \'etale map. Taking a resolution of singularities $Z \to C$, we form the commutative diagram
	\begin{equation}\label{eq:commonres}
	\begin{tikzcd}
	Z \ar[dr] \ar[drr,"q"] \ar[ddr,"g"'] & & \\
	& C \ar[r] \ar[d] & X \ar[d,"f"] \\
	& W \ar[r,"p"'] & Y.
	\end{tikzcd}
	\end{equation}
	Since all the maps $f,g,p,q$ are generically \'etale, there exists a dense open subset $U$ of $Y$ such that the restrictions $p^{-1}(U) \to U, f^{-1}(U) \to U$ and $g^{-1}(p^{-1}(U)) \to p^{-1}(U), q^{-1}(f^{-1}(U)) \to f^{-1}(U)$ are  \'etale. We check that $q^* \omega$ extends to a regular form on $Z$. This means that there exists some regular $2$-form (necessarily unique) on $Z$ whose restriction to some dense open subset, over which $q$ is \'etale, agrees with $q^* \omega$. Since $f$ is assumed Poisson, $f^*(\omega_0 |_U) = \omega |_{f^{-1}(U)}$. Therefore, 
	$$
	q^*(\omega |_{f^{-1}(U)}) = q^*(f^*(\omega_0 |_U)) =g^*(p^*(\omega_0 |_U)).
	$$
	Since $Y$ is assumed to have symplectic singularities, there exists a regular $2$-form $\eta$ on $W$ whose restriction to $p^{-1}(U)$ agrees with $p^*(\omega_0 |_U)$. Thus, $q^*(\omega |_{f^{-1}(U)}) |_V$ equals $g^*(\eta) |_{V}$, where $V = g^{-1}(p^{-1}(U)) \cap q^{-1}(f^{-1}(U))$ and $g^*(\eta)$ is a regular $2$-form on $Z$. 
\end{proof}

Instead of using $Y$ to deduce that $X$ has symplectic singularities, one can ask if we can use $X$ to deduce that $Y$ has symplectic singularities. As shown in the result below, the answer is yes, provided the morphism is also assumed proper; see also \cite[Proposition~2.4]{Beauville} or \cite[Lemma~6.12]{BellSchedQuiver}. The result is not required in this paper, but we provide a proof for completeness.

\begin{prop}\label{lem:keydominantsympconverse}
	Let $X,Y$ be complex normal Poisson varieties. Assume that $X$ has symplectic singularities and the Poisson structure on the smooth locus of $Y$ is non-degenerate. If there exists a generically \'etale proper Poisson morphism $f \colon X \to Y$ then $Y$ has symplectic singularities.   
\end{prop}

The outline of the proof is the same as that of Lemma~\ref{lem:keydominantsymp}. The difference is that we now have a meromorphic form $p^* \omega_0$ on $W$ that we wish to show is regular. Diagram \eqref{eq:commonres} implies that $g^*(p^* \omega_0) = q^*(f^* \omega)$ is regular on $Z$. We deduce from the key lemma below that $p^* \omega_0$ is regular.    

\begin{lem}\label{lem:meromorphickforms}
	Let $g \colon Z \rightarrow W$ be a proper, generically \'etale morphism between smooth complex varieties. Then a meromorphic $k$-form $\omega$ on $W$ is regular if and only if $g^* \omega$ is regular. 
\end{lem}

\begin{proof}
	Our assumptions imply that $g$ is surjective. First, we note that the locus where $\omega$ is not regular is a divisor on $W$; locally we can pick $w_1, \ds, w_n$ such that $d w_1 , \ds, d w_n$ are a basis of $\Omega_W^1$. Then $\omega$ can be uniquely expressed as $\sum_{\mathbf{i}} a_{\mathbf{i}} d w_{\mathbf{i}}$, and the non-regular locus of $\omega$ is the union of the non-regular loci of the meromorphic functions $a_{\mathbf{i}}$. 
	
	Next, we claim that the locus (on $W$) where $g$ is finite has complement of codimension at least two. Since $g$ is assumed proper, Stein factorization says that we can factor $g = \phi \circ h$, where $h \colon Z \rightarrow T$ has connected fibers and $\phi \colon T \rightarrow W$ is finite. It suffices then to show that the locus of points $t$ on $T$ where $\dim h^{-1}(t) = 0$ has complement $C$ of codimension at least two. But if $c$ is a generic point of an irreducible component $C_0$ of $C$, then $\dim C_0 + \dim h^{-1}(c) < \dim Z$ since $h^{-1}(C_0)$ is a proper closed subset of $Z$. Since $\dim h^{-1}(c) \ge 1$, this implies that $\dim C_0 < \dim T - 1$. 
	
	Therefore, we may assume that $g$ is a finite morphism. If $U \subset W$ is any open set such that $g$ is \'etale on $g^{-1}(U)$ then it is clear that $\omega |_U$ is regular if and only if $g^* (\omega |_U)$ is regular. Thus, we just need to consider a generic point $w \in g(R_g) \subset W$, where $R_g$ is the ramification divisor of $g$. Let $D \subset g(R_g)$ be an irreducible component, and $E \subset R_g$ an irreducible component of $R_g$ mapping onto $D$. We may assume that $w_1 = 0$ is a local equation for $D$ at $w$. Since $Z$ and $W$ are smooth, the local rings $\mc{O}_{W,D}$ and $\mc{O}_{Z,E}$ are (noetherian) regular local rings of dimension one; that is, they are discrete valuation rings. We have an embedding $g^* \colon \mc{O}_{W,D} \rightarrow \mc{O}_{Z,E}$ , with $\mc{O}_{Z,E}$ finite over $\mc{O}_{W,D}$. The function $w_1$ is a uniformizer for $\mc{O}_{W,D}$, and choosing a uniformizer $t$ for $\mc{O}_{Z,E}$, we have $g^*(w_1) = t^{\ell} u$ for some unit $u \in \mc{O}_{Z,E}^{\times}$. Here $\ell$ is the ramification index of $D$. The module $\Omega_{\mc{O}_{Z,E}}^1$ has basis $d t, d g^*(w_2), \ds, d g^*(w_n)$. Therefore, if $\omega_{\mathbf{i}} = a_{\mathbf{i}} d w_1 \wedge d w_{i_1} \wedge \cdots \wedge d w_{i_{k-1}}$ is a summand of $\omega$, for some $1 < i_1 < \cdots < i_{k-1} \le n$, then  
	$$
	g^* \omega_{\mathbf{i}} = \ell u g^*(a_{\mathbf{i}}) t^{\ell-1} d t \wedge d g^*(w_{i_1}) \wedge \cdots \wedge d g^*(w_{i_{k-1}}) + g^* (a_{\mathbf{i}}) t^{\ell} d u \wedge d g^*(w_{i_1}) \wedge \cdots \wedge d g^*(w_{i_{k-1}}).
	$$
	If $a_{\mathbf{i}} = w_1^{-r} s$ for some unit $s \in \mc{O}_{W,D}$ and $r \ge 1$, then $\ell u g^*(a_{\mathbf{i}}) t^{\ell-1} = \ell u^{1-r} g^*(s) t^{-\ell(r-1) - 1}$ has a pole of order $\ell(r-1) + 1 \ge 1$. Since $d u \in \bigoplus_{i \ge 2} \mc{O}_{Z,E} \,  d g^*(w_i)$, we deduce that $\omega_{\mathbf{i}}$ is regular if and only if $g^* \omega_{\mathbf{i}}$ is regular. 
\end{proof}

\subsection{Toric hyper-K\"ahler manifolds}

Consider a short exact sequence 
\begin{equation}\label{eq:exactintegerseq}
	0 \to \Z^k \stackrel{B}{\longrightarrow} \Z^n  \stackrel{A}{\longrightarrow} \Z^d \to 0. 
\end{equation}
We assume that no row of $B$ is zero. Write $T := \C^{\times}$ for the one torus. The above sequence encodes an action of $T^d$ on $\C^n$ via 
$$
(t_1, \ds, t_d) \cdot x_i = t_1^{a_{i,1}} \cdots t_d^{a_{i,d}} x_i.
$$
Since \eqref{eq:exactintegerseq} is exact, the stabilizer of any $x \in \C^n$, with $x_i \neq 0$ for all $i$, is trivial. In particular, the action is effective. The induced action on $T^* \C^n$ is Hamiltonian and we write $\mu \colon T^* \C^n \to \mf{t}_d^*$ for the associated moment map. Explicitly, 
\[
\mu(x,y) = \left( \sum_i a_{i,j} x_i y_i \right)_{j = 1}^d. 
\] If $\theta$ denotes a rational character of $T^d$ and $\zeta \in \mf{t}_d^*$ then we can take Hamiltonian reduction 
$$
\mc{M}_H(\theta,\zeta) := \mu^{-1}(\zeta)^{\theta} \git \, T^d. 
$$
Here $\mu^{-1}(\zeta)^{\theta}$ denotes the open subset of $\theta$-semistable points in $\mu^{-1}(\zeta)$. The variety $\mc{M}_H(\theta,\zeta)$ is a \textit{toric hyper-K\"ahler manifold} (also called a \textit{hypertoric variety} in the literature) and is a Higgs branch for the gauge theory $(T^d,\C^n)$. The fact that these varieties have symplectic singularities is well-known, but the proofs in the literature, \cite[Proposition~4.11]{HypertoricBB} or \cite[Theorem~2.16]{NagaokaUniversal}, always assume that the matrix $A$ is unimodular (so that the variety admits a symplectic resolution given by variation of GIT). Since we will need to consider matrices $A$ that are not unimodular, we explain how to extend this result to general toric hyper-K\"ahler manifolds.

\begin{lem}\label{lem:hypertoricbirational}
	Choose $\theta,\theta'$ such that $\mu^{-1}(\zeta)^{\theta'} \subset \mu^{-1}(\zeta)^{\theta}$. Then there exists a projective birational Poisson morphism $\mc{M}_H(\theta',\zeta) \to \mc{M}_H(\theta,\zeta)$.
\end{lem}

\begin{proof}
	Since we have assumed that no row of $B$ is zero, $\mu^{-1}(\zeta)$ is a reduced, irreducible complete intersection \cite[Lemma~4.7]{HypertoricBB}. Moreover, as shown in \cite[Proposition~4.11]{HypertoricBB} when $A$ is unimodular and in \cite{SVdBhypertoric} in general, the quotient $\mc{M}_H(\theta,\zeta)$ is normal. Since the variety is constructed as a Hamiltonian reduction, the Poisson bracket on $\mc{O}_{T^* \C^n}$ descends to a Poisson bracket on $\mc{O}_{\mc{M}_H(\theta,\zeta)}$. 
		
	The fact that there is a projective Poisson morphism $\pi \colon \mc{M}_H(\theta',\zeta) \to \mc{M}_H(\theta,\zeta)$ is a direct consequence of Hamiltonian reduction; see for instance the proof of \cite[Lemma~2.4]{BellSchedQuiver}. We need to check that it is birational. 
	
	Let $\mu^{-1}(\zeta)^{\theta\, \mr{st}}$ denote the set of $\theta$-stable points in $\mu^{-1}(\zeta)$ and $\mc{M}_H(\theta,\zeta)^{\theta\, \mr{st}}$ its image in $\mc{M}_H(\theta,\zeta)$. The map $\pi$ is bijective over $\mc{M}_H(\theta,\zeta)^{\theta\, \mr{st}}$. Hence, we need to show that $\mc{M}_H(\theta,\zeta)^{\theta\, \mr{st}}$ (or equivalently, $\mu^{-1}(\zeta)^{\theta\, \mr{st}}$) is non-empty. Since $\mu^{-1}(\zeta)^{0\, \mr{st}}$ is contained in $\mu^{-1}(\zeta)^{\theta\, \mr{st}}$, it suffices to show that $\mu^{-1}(\zeta)^{0\, \mr{st}} \neq \emptyset$. In other words, there exists a closed orbit in $\mu^{-1}(\zeta)$ with finite (in fact trivial) stabilizer. Let $U \subset T^* V$ consist of all points $(x,y)$ with $x_i,y_i \neq 0$ for all $1 \le i \le n$. As noted previously, the stabilizer of any point in $U$ is trivial. We claim that (a) every orbit in $U$ is closed in $T^* V$, and (b) $\mu^{-1}(\zeta) \cap U \neq \emptyset$. Thus, (a) and (b) would imply $\emptyset \neq \mu^{-1}(\zeta) \cap U  \subset \mu^{-1}(\zeta)^{0\, \mr{st}}$. 
	
	Let $(p,q) \in U$. If $p_i q_i =: \lambda_i \in \Cs$ then the equation $x_i y_i = \lambda_i$ holds for all points in $\overline{T^d \cdot (p,q)}$. But this forces $x_i,y_i \neq 0$ for all points $(x,y)$ in $\overline{T^d \cdot (p,q)}$. That is, $\overline{T^d \cdot (p,q)} \subset U$. Since all orbits in $U$ are free, we have $\overline{T^d \cdot (p,q)} = T^d \cdot (p,q)$ proving (a).  
	
	For (b), the exactness of \eqref{eq:exactintegerseq} implies that the rank of $A$ is $d$. Therefore, permuting the $x_i$, we may assume that the first $d \times d$ block of $A$ has non-zero determinant. Applying an automorphism to $T^d$ corresponds to multiplying $A$ on the left by a unimodular $d \times d$ matrix $U$. Therefore, replacing $A$ by $U A$, we may assume that $A$ is in Hermite form. In particular, the moment map relations $\mu(x,y) = \zeta$ become 
	\begin{equation}\label{eq:relHermite}
		x_i y_i = a_{i,i}^{-1} \zeta_i - \sum_{j > i} a_{i,i}^{-1} a_{j,i} x_j y_j.	
	\end{equation}
	Making further substitutions (and replacing $a_{i,i}^{-1} \zeta_i$ by some $\zeta_i'$), we may assume $a_{j,i} = 0$ for $j \le d$ in the relations \eqref{eq:relHermite}. The fact that no row of $B$ is zero translates into the fact that for each $1 \le i \le d$ there exists some $j > d$ with $a_{j,i} \neq 0$. This means that for generic $(x_{d+1},\ds,x_n,y_{d+1},\ds,y_n)$ with $x_j,y_j \neq 0$ the relations \eqref{eq:relHermite} can be satisfied, but only with $x_i y_i \neq 0$ for $1 \le i \le d$ too. Thus, $\mu^{-1}(\zeta) \cap U \neq \emptyset$. 
\end{proof}

\begin{prop}\label{prop:torichyperissymplectic}
	The toric hyper-K\"ahler manifold $\mc{M}_H(\theta,\zeta)$ has symplectic singularities. 
\end{prop}

\begin{proof}
	As noted in the proof of Lemma~\ref{lem:hypertoricbirational}, the quotient $\mc{M}_H(\theta,\zeta)$ is a normal Poisson variety. Moreover, since $T^d$ acts freely, with closed orbits, on the non-empty open set $\mu^{-1}(\zeta)^{\theta\, \mr{st}} \cap U$, the Poisson structure on $\mc{M}_H(\theta,\zeta)$ is generically non-degenerate. 
	
	
	Choose a generic $\theta'$ such that $\mu^{-1}(\zeta)^{\theta'} \subset \mu^{-1}(\zeta)^{\theta}$. Then Lemma~\ref{lem:hypertoricbirational} says that there exists a projective birational Poisson morphism $\mc{M}_H(\theta',\zeta) \to \mc{M}_H(\theta,\zeta)$. If $\mc{M}_H(\theta',\zeta)$ admits symplectic singularities then \cite[Lemma~6.12]{BellSchedQuiver} implies that  $\mc{M}_H(\theta,\zeta)$ will also admit symplectic singularities. Thus, we may assume that $\theta$ is generic.  
	
	To check that $\mc{M}_H(\theta,\zeta)$ has symplectic singularities, it suffices to check \'etale locally. As explained in the proof of \cite[Proposition~6.2]{HS}, the fact that $\theta$ is generic means that the stabilizer under $T^d$ of each point in $\mu^{-1}(\zeta)^{\theta}$ is finite. Therefore, the (\'etale) symplectic slice theorem e.g. \cite[Theorem~3.8]{BellSchedQuiver}\footnote{This is stated and proved for Nakajima quiver varieties, but both the statement and proof go through without change for any reductive group acting symplectically on a symplectic vector space.}, says that \'etale locally $\mc{M}_H(\theta,\zeta)$ is isomorphic (as a Poisson variety) to the quotient of a symplectic vector space by a finite (abelian) group acting symplectically. In particular, it has symplectic singularities by \cite[Proposition~2.4]{Beauville} and hence the Poisson structure is non-degenerate on the whole of the smooth locus of $\mc{M}_H(\theta,\zeta)$.  
\end{proof}

Whilst this note was in preparation, the above statement also appeared as \cite[Proposition~5.1]{BielawskiFoscolo}. 

\subsection{Coulomb branches}

Coulomb branches are normal varieties whose smooth locus admits a symplectic form \cite{BFNII}. Let $G^{\circ}$ be the connected component of the identity in $G$. Then, as noted in \cite[Remarks~2.8(3)]{BFNII}, $\mc{M}_C(G,N) \cong \mc{M}_C(G^{\circ},N)/ ( G / G^{\circ})$ as Poisson varieties. Hence $\mc{M}_C(G,N)$ will have symplectic singularities by \cite[Proposition~2.4]{Beauville} if we can show that $\mc{M}_C(G^{\circ},N)$ has symplectic singularities. Therefore, we may assume that $G$ is connected. We first consider the abelian case.

\begin{lem}\label{lem:abeilanCoulomb}
Assume $G = T^k$ is a torus. Then $\mc{M}_C(G,N)$ has symplectic singularities. 
\end{lem}

\begin{proof}
	The action of $G = T^k$ on $N = \C^m$ is encoded in an integral $k \times m$ matrix $B_0$. Namely, 
	$$
	(t_1, \ds, t_k) \cdot x_i = t_1^{b_{1,i}} \cdots t_d^{b_{k,i}} x_i.
	$$  
	If we decompose $N = N_0 \oplus N^{G}$, then \cite[3(vii)]{BFNII} says that $\mc{M}_C(G,N) = \mc{M}_C(G,N_0)$. Therefore we may assume that $N = N_0$. In other words, no row of $B_0$ is zero; this will be important later. 
	
	The idea of course is to identify the Coulomb branch with a toric hyper-K\"ahler manifold and apply Proposition~\ref{prop:torichyperissymplectic}. However, this identification only holds if there is sufficient matter in the theory, specifically if the representation $N$ is assumed to be a faithful $T$-module. 
	
	Let $N' = N \oplus \C^k = \C^n$, where $T^k$ acts on $\C^k$ in the natural way (so that the weights are encoded by the identity matrix) and $n = m + k$. By \cite[4(vi)]{BFNII}, there is a birational Poisson morphism $\mc{M}_C(T^k,N) \to \mc{M}_C(T^k,N')$. Then Lemma~\ref{lem:keydominantsymp} says that $\mc{M}_C(T^k,N)$ will have symplectic singularities if we can show that $\mc{M}_C(T^k,N')$ has symplectic singularities. The action of $T^k$ on $N'$ is encoded in the matrix $B = \left(\begin{array}{c}
	B_0 \\
	\mr{Id} 
	\end{array} \right)$ and we may form a short exact sequence 
	$$
	0 \to \Z^k \stackrel{B}{\longrightarrow} \Z^n  \stackrel{A}{\longrightarrow} \Z^d \to 0, 
	$$
	where $A = (\mr{Id} | -\!B_0^T)$. We note that no row of $B$ is zero. In this situation, it is noted in \cite[4(iv)]{BFNII} that $\mc{M}_C(T, N')$ is isomorphic to the affine toric hyper-K\"ahler manifold $\mc{M}_H((T^d)^{\vee},N')$. By Proposition~\ref{prop:torichyperissymplectic}, the latter has symplectic singularities.
\end{proof}

Now we return to the general situation, where $G$ is a connected reductive group. Let $T$ be a maximal torus of $G$ and $W$ the associated Weyl group. It is shown in \cite[Lemma~5.9, Lemma~5.10]{BFNII} that there exists a birational Poisson morphism $\mc{M}_C(G,N) \to \mc{M}_C(T, N |_T)/W$. Lemma~\ref{lem:abeilanCoulomb} implies that $\mc{M}_C(T, N |_T)$ has symplectic singularities. It follows from \cite[Proposition~2.4]{Beauville} that $\mc{M}_C(T, N |_T)/W$ also has symplectic singularities. Therefore Theorem~\ref{thm:rankonemain} is a consequence of Lemma~\ref{lem:keydominantsymp}. 

\section*{Acknowledgements}

We would like to thank Dinakar Muthiah, Hiraku Nakajima, Oded Yacobi and Alex Weekes for stimulating discussions about Coulomb branches. We also thank the referee for comments that improved the article. The author was partially supported by a Research Project Grant from the Leverhulme Trust and by the EPSRC grant EP-W013053-1. 

On behalf of all authors, the corresponding author states that there is no conflict of interest.

\def\cprime{$'$} \def\cprime{$'$} \def\cprime{$'$} \def\cprime{$'$}
\def\cprime{$'$} \def\cprime{$'$} \def\cprime{$'$} \def\cprime{$'$}
\def\cprime{$'$} \def\cprime{$'$} \def\cprime{$'$} \def\cprime{$'$}
\def\cprime{$'$} \def\cprime{$'$}


\end{document}